\documentclass[11pt]{article}
\usepackage[utf8]{inputenc}
\usepackage[T1]{fontenc}
\usepackage[margin=1in]{geometry}
\usepackage{amsmath,amssymb,amsthm}
\usepackage{mathtools}
\usepackage[unicode,pdfusetitle]{hyperref}

\newcommand{\GA}{\mathrm{GA}}
\newcommand{\Cr}{\mathrm{Cr}}
\newcommand{\Id}{\mathrm{Id}}
\newcommand{\Aut}{\operatorname{Aut}}
\newcommand{\F}{\mathbb{F}}

\theoremstyle{plain}
\newtheorem{theorem}{Theorem}[section]
\newtheorem{proposition}[theorem]{Proposition}
\newtheorem{lemma}[theorem]{Lemma}
\newtheorem{corollary}[theorem]{Corollary}
\theoremstyle{definition}
\newtheorem{definition}[theorem]{Definition}
\newtheorem{example}[theorem]{Example}
\theoremstyle{remark}
\newtheorem{remark}[theorem]{Remark}

\title{Order and Pascal depth of Pascal finite automorphisms of the plane}
\author{%
  El\.zbieta Adamus \\ Faculty of Applied Mathematics, \\
  AGH University of Krakow \\ al.~Mickiewicza 30, 30-059 Krak\'ow, Poland \\
  e-mail: esowa@agh.edu.pl
  \\[0.6cm]
  Zbigniew Hajto \\ Faculty of Mathematics and Computer Science, \\
  Jagiellonian University \\ ul.~\L ojasiewicza 6, 30-348 Krak\'ow, Poland \\
  e-mail: zbigniew.hajto@uj.edu.pl%
}
\date{}

\begin{document}
\maketitle

\begin{abstract}
Let $K$ be a field of characteristic $p>0$. For a Pascal finite automorphism
$F$ of the affine plane we show that its order is determined by its Pascal depth,
$|F|=p^{\lceil\log_p\tau_K(F)\rceil}$, and that, combined with Dolgachev's theorem
on the plane Cremona group, this pins the order spectrum of Pascal finite plane
automorphisms to $\{1,p,p^2\}$ and bounds the Pascal depth by $\tau_K(F)\le p^2$.
For the polynomial group $\GA_2(K)$ we give a second, independent proof of the
order-$p^2$ ceiling, a purely group-theoretic argument from the Jung--van der
Kulk amalgam and Serre's tree theorem, using no birational geometry.
We prove that the bound is sharp in two independent senses. Order~$p^2$ is
attained by the length-two Witt vectors, and Pascal depth $p^2$ is attained by an
explicit tame automorphism $G_p$, for which we give a characteristic-free
proof that $\tau_K(G_p)=p^2$. We  contrast the plane
with higher dimensions, where both order and depth are unbounded.
\end{abstract}



\section{Introduction}\label{sec:intro}

Throughout, $K$ is a field, $K[X]=K[X_1,\dots,X_n]$, and $\GA_n(K)=\Aut_KK[X]$ is
the group of polynomial automorphisms of $K^n$. Over a field of characteristic
zero the \emph{Exponential Generators Conjecture} asserts that $\GA_n(K)$ is
generated by the affine automorphisms together with the exponentials $\exp(D)$ of
locally nilpotent derivations~\cite{E}. In positive characteristic the
exponential map is unavailable, and \emph{Pascal finite} automorphisms take its
place~\cite{ABCH2}. The resulting Pascal Finite Generators Conjecture agrees with
the exponential one in characteristic zero, where the two classes of generators
coincide.

The order of a Pascal finite automorphism in characteristic $p$ is always a power
of $p$~\cite{ABCH2}. Our first observation (Lemma~\ref{lem:order-depth}) is that
this power is governed exactly by the Pascal depth,
\[
   |F|=p^{\lceil\log_p\tau_K(F)\rceil}.
\]
In dimension $n\ge 3$ the depth, and hence the order, is unbounded
(Proposition~\ref{prop:chains}). The plane is different. Viewing
$\GA_2(K)\subset\Cr_2(K)$ inside the plane Cremona group and invoking Dolgachev's
theorem that $\Cr_2(K)$ has no element of order $p^3$~\cite{Do,Se2}, we obtain
the following.

\begin{theorem}\label{thm:tower}
Let $\operatorname{char}K=p>0$ and let $F\in\GA_2(K)$ be Pascal finite. Then
$|F|=p^{\lceil\log_p\tau_K(F)\rceil}\in\{1,p,p^2\}$ and $\tau_K(F)\le p^2$.
Equivalently,
\[
   |F|=1\iff\tau_K(F)=1,\quad
   |F|=p\iff 2\le\tau_K(F)\le p,\quad
   |F|=p^2\iff p+1\le\tau_K(F)\le p^2 .
\]
Both nontrivial values $p$ and $p^2$ occur, and no higher order occurs.
\end{theorem}

For the polynomial group $\GA_2(K)$ the ceiling $|F|\mid p^2$ admits, in
addition, a proof independent of Dolgachev's theorem, from the Jung--van der
Kulk amalgam and Serre's theorem on group actions on trees
(Proposition~\ref{prop:amalgam}). 

The bound $\tau_K(F)\le p^2$ is sharp, and in two genuinely different ways.
On the level of order, the value $p^2$ is realised by the length-two Witt
vectors (Proposition~\ref{prop:witt}). On the level of depth, order
$p^2$ alone does not force $\tau_K=p^2$. It only forces $\tau_K\in\{p+1,\dots,
p^2\}$, so depth-sharpness requires a dedicated example.

\begin{theorem}\label{thm:Gp}
Let $\operatorname{char}K=p>0$ and set
\[
   M_p=(X_1+X_2^{p-1},\,X_2+1),\qquad S_p=(X_2+X_1^{2p-1},\,X_1),\qquad
   G_p=S_p\circ M_p\circ S_p^{-1}\in\GA_2(K).
\]
Then $G_p$ is tame, is not a quasi-translation, and $\tau_K(G_p)=|G_p|=p^2$.
\end{theorem}

The paper is organised as follows. Section~\ref{sec:prelim} recalls the
 Pascal finite formalism. Section~\ref{sec:order-depth} proves the
order--depth identity and the unboundedness in dimension $n\ge 3$.
Section~\ref{sec:tower} proves Theorem~\ref{thm:tower} from
Dolgachev's theorem. Section~\ref{sec:sharp} treats sharpness.
Section~\ref{sec:disc} collects the structural interpretation.

\section{The Pascal finite maps}\label{sec:prelim}

An element $F\in\GA_n(K)$ is written
$F=(F_1,\dots,F_n)$ with $F_i\in K[X]$, composition being
$(F\circ G)_i=F_i\circ G$.
We write $F^{(m)}$ for the $m$-fold composite and
$|F|$ for the order of $F$ in $\GA_n(K)$. The affine automorphisms $X\mapsto
AX+b$, with $A\in\mathrm{GL}_n(K)$ and $b\in K^n$, form the subgroup
$\mathrm{Aff}(K,n)$. We call $F$ \emph{triangular} if
$F_i=\lambda_iX_i+q_i(X_{i+1},\dots,X_n)$ with $\lambda_i\in K^\times$ and
$q_i\in K[X_{i+1},\dots,X_n]$, so that $q_n\in K$; the triangular automorphisms
form the \emph{de Jonqui\`eres} subgroup $J(K,n)$, and we put
$B(K,n)=\mathrm{Aff}(K,n)\cap J(K,n)$. A triangular $F$ with
$\lambda_1=\dots=\lambda_n=1$, i.e.\ $F_i=X_i+q_i(X_{i+1},\dots,X_n)$, is called
\emph{unitriangular}. Finally, $F$ is \emph{elementary} if, after a permutation
of the coordinates, $F=(X_1+q(X_2,\dots,X_n),X_2,\dots,X_n)$ for some
$q\in K[X_2,\dots,X_n]$, and \emph{tame} if it lies in the subgroup of $\GA_n(K)$
generated by $\mathrm{Aff}(K,n)$ and the elementary automorphisms.

\medskip

For $F\in K[X]^n$ let $\sigma_F\colon K[X]\to K[X]$ be the pullback
$\sigma_F(P)=P\circ F$, and put
\[
   \Delta_F:=\sigma_F-\Id .
\]
When $F$ is an automorphism, $\sigma_F$ is a $K$-algebra automorphism and
$\Delta_F$ is a $\sigma_F$-derivation, i.e. $\Delta_F(PQ)=\Delta_F(P)\,\sigma_F(Q)
+P\,\Delta_F(Q)$. Since $\sigma_F^{m}(X_i)=F^{(m)}_i$, the $m$-fold composite
$F^{(m)}$ is recovered by the Newton expansion
\begin{equation}\label{eq:newton}
   F^{(m)}=\sigma_F^m(X)=(\Id+\Delta_F)^m(X)=\sum_{k\ge 0}\binom{m}{k}\Delta_F^k(X).
\end{equation}

The algorithm for $F$ described in~\cite{ABCH}, applied to $g\in K[X]$, forms
the sequence $g$, $\Delta_F(g)$, $\Delta_F^2(g)$, \dots, each term obtained from
its predecessor by the substitution $h\mapsto h\circ F-h$.

\begin{definition}\label{def:pf}
The \emph{Pascal depth of $g$ w.r.t. $F$} is the
number of steps after which the algorithm stops,
\[
   \tau_K^F(g)=\min\{m\ge 1:\Delta_F^m(g)=0\}
\]
(and $\tau_K^F(g)=\infty$ if it never halts). We set
$\tau_K(F)=\max_{1\le i\le n}\tau_K^F(X_i)$ and call $F$ \emph{Pascal finite} when
the algorithm halts on every coordinate, i.e. $\tau_K(F)<\infty$. The class of
Pascal finite automorphisms of $K^n$ is denoted $\mathrm{PF}(K,n)$. 
\end{definition}

In characteristic zero the Pascal finite automorphisms are exactly the
exponential ones~\cite{ABCH2}.

\begin{remark}\label{rem:canonical}
In characteristic zero the Pascal depth admits several equivalent readings. Namely, the
operator $\Delta_F$ is locally nilpotent and $F=\exp(D)$ for the derivation
\[
   D=\log\sigma_F=\sum_{k\ge1}(-1)^{k-1}\Delta_F^{k}/k,
\]
equivalently, the minimal
polynomial of $\sigma_F$ on each invariant finite-dimensional subspace is
$(T-1)^{\tau_K}$, equivalently, the iterates $F^{(m)}$ obey a polynomial-type
linear recurrence. In characteristic $p>0$ these descriptions degenerate. The
logarithm needs the forbidden denominators $k!$, and the recurrence and
minimal-polynomial machinery break down, whereas the difference algorithm, using only substitution, holds. The Frobenius identity
$\sigma_F^{p^{s}}=\Id+\Delta_F^{p^{s}}$ of Lemma~\ref{lem:order-depth} shows it sharpens.
Stopping of the algorithm is equivalent to
$p$-primary torsion. We therefore adopt the difference algorithm condition as the  definition throughout.
\end{remark}

\section{Order is governed by depth}\label{sec:order-depth}

\begin{lemma}\label{lem:order-depth}
Let $\operatorname{char}K=p>0$ and let $F\in\mathrm{PF}(K,n)$. Then
\[
   |F|=p^{\lceil\log_p\tau_K(F)\rceil}.
\]
In particular the order of a Pascal finite automorphism is a power of $p$.
\end{lemma}

\begin{proof}
Because $\Id$ and $\Delta_F$ commute, the Frobenius identity in characteristic $p$
gives $\sigma_F^{p^{s}}=(\Id+\Delta_F)^{p^{s}}=\Id+\Delta_F^{p^{s}}$, the middle
binomial coefficients $\binom{p^{s}}{k}$ ($0<k<p^{s}$) vanishing modulo $p$ by
Lucas' theorem~\cite{L}. Applying this to the coordinates,
$F^{(p^{s})}=X+\Delta_F^{p^{s}}(X)$, so $F^{(p^{s})}=\Id$ if and only if
$\Delta_F^{p^{s}}(X_i)=0$ for all $i$, i.e. if and only if $p^{s}\ge\tau_K(F)$.
The least such $s$ is $\lceil\log_p\tau_K(F)\rceil$, and as $F^{(p^{s})}=\Id$ for
that $s$, the order $|F|$ divides $p^{s}$ and is therefore a power of $p$.
Minimality of $s$ gives $|F|=p^{s}$.
\end{proof}

\begin{lemma}\label{lem:unitri}
Let $\operatorname{char}K=p>0$. Every unitriangular $F\in\GA_n(K)$ is Pascal
finite.
\end{lemma}

\begin{proof}
Write $R_i=K[X_i,\dots,X_n]$ and $R_{n+1}=K$. Each $R_i$ is $\sigma_F$-stable,
hence $\Delta_F$-stable. We show by descending induction on $i$ that
$\sigma_F^{p^{s}}|_{R_i}=\Id$ for some $s$. 
For $i=1$ this gives
$\Delta_F^{p^{s}}(X_j)=\sigma_F^{p^{s}}(X_j)-X_j=0$ for all $j$ by the Frobenius
identity of Lemma~\ref{lem:order-depth}, i.e.\ $\tau_K(F)\le p^{s}<\infty$. The
case $i=n+1$ is trivial. Assume $\sigma_F^{p^{s}}|_{R_{i+1}}=\Id$. By the
Frobenius identity, $\Delta_F^{p^{s}}|_{R_{i+1}}=0$, and
\[
   \sigma_F^{p^{t}}(X_i)=X_i+\Delta_F^{p^{t}}(X_i)
   =X_i+\Delta_F^{p^{t}-1}(q_i),\qquad q_i\in R_{i+1},
\]
which vanishes as soon as $p^{t}-1\ge p^{s}$, say for $t=s+1$. Thus
$\sigma_F^{p^{s+1}}$ fixes $X_i$, and $\sigma_F^{p^{s+1}}=(\sigma_F^{p^{s}})^{p}$
fixes $R_{i+1}$ pointwise. As $R_i=R_{i+1}[X_i]$, it is the identity on $R_i$.
\end{proof}

The exponent $\lceil\log_p\tau_K(F)\rceil$ depends on $F$ and cannot be bounded
uniformly. In dimension $n\ge 3$ the Pascal depth, and hence the order, is
unbounded.

\begin{proposition}\label{prop:chains}
Let $\operatorname{char}K=2$ and let
$F_n=(X_1+X_2^2,\,X_2+X_3^2,\,\dots,\,X_{n-1}+X_n^2,\,X_n)\in\GA_n(K)$. Then
$F_n$ is unitriangular, hence Pascal finite by Lemma~\ref{lem:unitri}, and
$|F_n|=2^{\lceil\log_2 n\rceil}$.
In particular, for $n=2^{r-1}+1$ the order is $2^{r}$, with $r$ arbitrary.
\end{proposition}

\begin{proof}
The map is unitriangular, so Pascal finite by Lemma~\ref{lem:unitri}. In
characteristic $2$ the Frobenius
identity gives
$F_n^{(2)}(X_i)=X_i+X_{i+1}^2+(X_{i+1}+X_{i+2}^2)^2=X_i+X_{i+2}^{4}$, and
inductively $F_n^{(2^{j})}(X_i)=X_i+X_{i+2^{j}}^{2^{2^{j}}}$, with the convention
$X_k=0$ for $k>n$. Hence $F_n^{(2^{j})}=\Id$ if and only if $2^{j}\ge n$. As the
order is a power of $2$ by Lemma~\ref{lem:order-depth}, $|F_n|=2^{\lceil\log_2
n\rceil}$.
\end{proof}

Thus for $n\ge 3$ no analogue of Theorem~\ref{thm:tower} can hold.

\section{Dolgachev's theorem and the order tower}\label{sec:tower}

The plane Cremona group $\Cr_2(K)$ is the group of $K$-automorphisms of the
rational function field $K(X_1,X_2)$. It contains $\GA_2(K)$, so every restriction
on torsion in $\Cr_2(K)$ applies a fortiori to polynomial automorphisms. We use
the following.

\begin{theorem}[Dolgachev~\cite{Do}; see also {\cite[Th.~3.5]{Se2}}]\label{thm:dolg}
If $\operatorname{char}K=p>0$, then $\Cr_2(K)$ contains no element of order $p^3$.
\end{theorem}

The proof rests on the classification of minimal rational surfaces and has no
counterpart for $n\ge 3$.

\begin{proof}[Proof of Theorem~\ref{thm:tower}]
By Lemma~\ref{lem:order-depth}, $|F|=p^{\lceil\log_p\tau_K(F)\rceil}$ is a power
of $p$. Since $F\in\GA_2(K)\subset\Cr_2(K)$, Theorem~\ref{thm:dolg} forbids order
$p^3$, hence any order $p^{r}$ with $r\ge 3$ (such an element would generate a
cyclic group with an element of order $p^3$). Therefore $|F|\mid p^2$, i.e.
$|F|\in\{1,p,p^2\}$ and $\tau_K(F)\le p^2$. Moreover, order $p$ is realised by any nontrivial transvection
$(X_1+\lambda X_2,X_2)$, and order $p^2$ by the length-two Witt vectors
(Proposition~\ref{prop:witt}) as well as by $G_p$ (Theorem~\ref{thm:Gp}). 
\end{proof}

\begin{corollary}\label{cor:depthbound}
Every Pascal finite $F\in\GA_2(K)$ satisfies $\tau_K(F)\le p^2$. 
\end{corollary}

This bound is
not a consequence of the difference calculus. 
In arbitrary dimension no such bound holds (Proposition~\ref{prop:chains}). In
the plane it has two independent sources --- the birational geometry of
surfaces, valid for all of $\Cr_2(K)$ (Theorem~\ref{thm:dolg}), and, for the
polynomial group $\GA_2(K)$ alone, the amalgamated structure
(Proposition~\ref{prop:amalgam} below).

\begin{proposition}\label{prop:amalgam}
Let $\operatorname{char}K=p>0$. Every element of $\GA_2(K)$ of finite $p$-power
order has order dividing $p^2$. Equivalently, by Lemma~\ref{lem:order-depth},
every Pascal finite $F\in\GA_2(K)$ satisfies $\tau_K(F)\le p^2$. 
\end{proposition}

\begin{proof}
By the Jung--van der Kulk theorem~\cite{J,vdK}, in the form recalled
in~\cite{AH}, $\GA_2(K)$ is the amalgamated free product
$\GA_2(K)=\mathrm{Aff}(K,2)*_{B(K,2)}J(K,2)$, where
$B(K,2)=\mathrm{Aff}(K,2)\cap J(K,2)$. It acts
without inversion on its Bass-Serre tree, the vertex stabilisers being the
conjugates of the two factors. A finite subgroup of such a group fixes a vertex
(Serre~\cite{SeTr}). When applied to the cyclic group generated by a torsion element,
this makes every finite-order $F\in\GA_2(K)$ conjugate into $\mathrm{Aff}(K,2)$
or into $J(K,2)$. Since order is a conjugacy invariant, it suffices to bound the
$p$-power torsion of each factor.

Over any field $L$ of
characteristic $p$, the only $p$-power root of unity is $1$, since
$\lambda^{p^{s}}=1$ gives $(\lambda-1)^{p^{s}}=\lambda^{p^{s}}-1=0$, whence
$\lambda=1$.

\emph{The affine factor.} Embed
$\mathrm{Aff}(K,2)\hookrightarrow\mathrm{GL}_3(K)$ by sending the map
$X\mapsto AX+b$ to $\left(\begin{smallmatrix}A&b\\0&1\end{smallmatrix}\right)$.
Let $\Id+N$ be the image of an element of $p$-power order, and compute its
eigenvalues in an algebraic closure of $K$. They are $p$-power roots of unity,
hence all equal to $1$. Thus $N\in M_3(K)$ is nilpotent,
of index $\nu:=\operatorname{ind}N\le3$. The Frobenius identity
$(\Id+N)^{p^{s}}=\Id+N^{p^{s}}$ gives $(\Id+N)^{p^{s}}=\Id\iff p^{s}\ge\nu$, so
the order is $p^{\lceil\log_p\nu\rceil}\mid p^{\lceil\log_p3\rceil}\mid p^2$. For
$p\ge3$ one has $\lceil\log_p3\rceil=1$, so the affine factor contributes only
order $p$. It reaches $p^2$ solely at $p=2$, and then only for $\nu=3$, i.e.\ for
$N^2\ne0$. An instance is $(X_1+X_2,\,X_2+1)$ over a field of characteristic $2$.

\emph{The de Jonqui\`eres factor.} Write $E\in J(K,2)$ in the normal form
of~\cite{AH}, $E=(\lambda_1X_1+a_1(X_2),\,\lambda_2X_2+a_2)$ with
$\lambda_1,\lambda_2\in K^\times$, $a_1\in K[X_2]$, $a_2\in K$. An induction on
$m$ shows that $E^{(m)}$ has the same shape, with $\lambda_i$ replaced by
$\lambda_i^{m}$. If $E$ has $p$-power order, say $E^{(p^{s})}=\Id$, then
$\lambda_1^{p^{s}}=\lambda_2^{p^{s}}=1$, so $\lambda_1=\lambda_2=1$ by the
observation above. Thus $E=(X_1+a_1(X_2),\,X_2+a_2)$, and
\[
   E^{(p)}=\Bigl(X_1+\textstyle\sum_{j=0}^{p-1}a_1(X_2+ja_2),\ X_2\Bigr)
\]
is elementary (since $pa_2=0$), so
$E^{(p^2)}=(E^{(p)})^{(p)}=\Id$ and $|E|\mid p^2$. 
The two factors of $p$ correspond to the two steps of this computation. $E^{(p)}$
restores the second coordinate, and $(E^{(p)})^{(p)}$ the first. The
map $M_p$ of Section~\ref{sec:sharp}, with $\tau_K(M_p)=p+1$, realises the bound.

Both factors give order dividing $p^2$, hence so does every $p$-power-order
element of $\GA_2(K)$.
\end{proof}

The bound in Corollary \ref{cor:depthbound} uses
only the amalgamated structure of $\GA_2(K)$ and Serre's tree theorem. It is
independent of Dolgachev's theorem and of the birational geometry of surfaces.

\begin{remark}\label{rem:generation}
Proposition~\ref{prop:amalgam} is an \emph{element-wise} bound, and is not the
generation statement of~\cite{AH}. There it is shown that the group
$\{F\in\GA_2(K):F(0)=0,\ \det J_F=1\}$ is \emph{generated} by elements of order
dividing $p^2$. 
A product of such generators can have arbitrarily large
$p$-power order.
What~\cite{AH} contributes to the present proof is the
amalgam in factor-normal form and the per-factor order computation
$E^{(p^2)}=\Id$. The passage to the individual ceiling is what is added here.
\end{remark}

\section{Sharpness}\label{sec:sharp}

A polynomial automorphism $F=\Id+H$ with $H\in K[X]^n$ is a
\emph{quasi-translation} if $\Id-H$ is its inverse, equivalently if $H\circ F=H$.
Over a field of arbitrary characteristic this is equivalent to
$\Delta_F^2(\Id)=0$. When $\deg H=0$ one recovers an ordinary translation.  In
characteristic $p$ this forces $|F|\mid p$. 
Indeed, 
Lemma~\ref{lem:order-depth} gives $|F|=p^{\lceil\log_p 2\rceil}=p$ (or $1$).
In particular a quasi-translation
cannot have Pascal depth above $2$.

\begin{example}\label{ex:psi}
Let $\operatorname{char}K=p>0$, $w=X_3X_1+X_2^p$, and
$\Psi_p=(X_1-X_3^{p-1}w^{p},\,X_2+X_3w)\in\GA_2(K(X_3))$. By the Frobenius identity
$(X_2+X_3w)^{p}=X_2^{p}+X_3^{p}w^{p}$, so
\[
   \Psi_p(w)=X_3\bigl(X_1-X_3^{p-1}w^{p}\bigr)+(X_2+X_3w)^{p}=X_3X_1+X_2^{p}=w,
\]
i.e. $w$ is $\Psi_p$-invariant. Hence $\Delta_{\Psi_p}(X_1)=-X_3^{p-1}w^{p}$ and
$\Delta_{\Psi_p}(X_2)=X_3w$ are themselves invariant, so $\Delta_{\Psi_p}^2(\Id)
=0$. Hence $\tau_K(\Psi_p)=2$ and $|\Psi_p|=p$. In
particular $\Psi_p$ does not realise depth $p^2$.
\end{example}

\begin{proposition}\label{prop:witt}
Let $\operatorname{char}K=p>0$. The additive group $W_2(K)$ of length-two Witt
vectors, an affine plane on which $W_2(K)$ acts by translations, embeds as
$W_2(K)\hookrightarrow\GA_2(K)\subset\Cr_2(K)$, and its nonzero elements have
order $p^2$. 
\end{proposition}

\begin{proof}
A length-two Witt vector has additive order $p^2$~\cite[Ch.~II, \S6]{SeCL},
giving the embedding and the
order. Such an element is unipotent of finite $p$-power order, hence
$\Delta_F^{p^2}=0$ by the identity in Lemma~\ref{lem:order-depth}, so it is Pascal
finite. Dolgachev's theorem (Theorem~\ref{thm:dolg}) shows $p^2$ is maximal.
\end{proof}

Consequently $\GA_2(K)$ contains Pascal finite automorphisms of order
$p^2$, and the bound of Theorem~\ref{thm:tower} cannot be lowered to $p$.
Proposition~\ref{prop:witt} settles sharpness of the order. 
It does not settle sharpness of the depth. An order-$p^2$ map has $\tau_K\in\{p+1,
\dots,p^2\}$. We present a map with
$\tau_K=p^2$.
We keep the notation of Theorem~\ref{thm:Gp}.

The map
$M_p=(X_1+X_2^{p-1},X_2+1)$ has order $p^2$ but only $\tau_K(M_p)=p+1$. 
The Pascal depth, unlike order, is not preserved under nonlinear conjugation.
Take $S_p=(X_2+X_1^{2p-1},X_1)$. It has the inverse
$S_p^{-1}=(X_2,\,X_1-X_2^{2p-1})$. The composition
$G_p=S_p\circ M_p\circ S_p^{-1}$ is of the form 
\begin{equation}\label{eq:Gp-explicit}
   G_p=\Bigl(\,X_1-X_2^{2p-1}+1+\bigl(X_2+(X_1-X_2^{2p-1})^{p-1}\bigr)^{2p-1},
   \;\;X_2+(X_1-X_2^{2p-1})^{p-1}\,\Bigr).
\end{equation}
For $p=2$ we obtain a map over $\F_2$
\[
   G_2=\bigl(X_1^{3}+X_1^{2}X_2^{3}+X_1^{2}X_2+X_1X_2^{6}+X_1X_2^{2}+X_1
   +X_2^{9}+X_2^{7}+X_2^{5}+1,\;\;X_1+X_2^{3}+X_2\bigr),
\]
a tame automorphism of degree $9$ and order $4$ with Pascal depth
$\tau_K(G_2)=4$.

\begin{lemma}\label{lem:conj}
For $T\in\GA_n(K)$ and any $F$, $\Delta_{TFT^{-1}}=\sigma_{T^{-1}}\circ\Delta_F
\circ\sigma_{T}$, hence 
\[ \tau_K^{TFT^{-1}}(X_i)=\tau_K^{F}(T_i) \quad \mathrm{and} \quad \tau_K(TFT^{-1})=\max_i\tau_K^{F}(T_i).\]
\end{lemma}

\begin{proof}
From $\sigma_{TFT^{-1}}=\sigma_{T^{-1}}\sigma_F\sigma_{T}$ we get
\[
   \Delta_{TFT^{-1}}=\sigma_{T^{-1}}(\sigma_F-\Id)\sigma_T
   =\sigma_{T^{-1}}\Delta_F\sigma_T,
\]
so $\Delta_{TFT^{-1}}^k(X_i)=\sigma_{T^{-1}}\bigl(\Delta_F^k(T_i)\bigr)$. As $\sigma_{T^{-1}}$ is injective,
this vanishes iff $\Delta_F^k(T_i)=0$.
\end{proof}

\begin{proof}[Proof of Theorem~\ref{thm:Gp}]
By the Jung--van der Kulk theorem~\cite{J,vdK} every element of $\GA_2(K)$ is tame. 
$S_p$ and $M_p$ are tame, so $G_p$ is tame. Since order is a conjugacy invariant, $|G_p|=|M_p|$.
Iterating $M_p$ gives $\sigma_{M_p}^{n}(X_2)=X_2+n$ and
$\sigma_{M_p}^{n}(X_1)=X_1+P_n$ with $P_n:=\sum_{j=0}^{n-1}(X_2+j)^{p-1}$, whence
$\tau_K(M_p)=p+1$ and $|M_p|=p^{\lceil\log_p(p+1)\rceil}=p^2$ by
Lemma~\ref{lem:order-depth}.

By Lemma~\ref{lem:conj}, $\tau_K(G_p)=\max\{\tau_K^{M_p}(X_2+X_1^{2p-1}),\,
\tau_K^{M_p}(X_1)\}$. Write $\sigma:=\sigma_{M_p}$, $\Delta:=\Delta_{M_p}$.
Here $\Delta(X_2)=1$ and $\Delta^2(X_2)=0$, and $\tau_K^{M_p}(X_1)=p+1$. Since
$\Delta$ is $K$-linear, $\Delta^{m}(X_2+X_1^{2p-1})=\Delta^{m}(X_1^{2p-1})$ for $m\ge2$. 
Hence,
$\tau_K^{M_p}(X_2+X_1^{2p-1})=\tau_K^{M_p}(X_1^{2p-1})$ for $\tau_K^{M_p}(X_1^{2p-1})\ge2$. It therefore suffices to
prove $\tau_K^{M_p}(X_1^{2p-1})=p^2$. Once this holds,
$\tau_K(G_p)>2$, so $G_p$ is not a quasi-translation.

\emph{Upper bound.} $\Delta^{p^2}=(\sigma-\Id)^{p^2}=\sigma^{p^2}-\Id=0$ because
$|M_p|=p^2$.

\emph{Lower bound.} For $p=2$ one checks $\tau_K^{M_2}(X_1^{3})=4$
directly.  

Take $p$ odd. Since $p^2-1=(p-1)+(p-1)p$ with $0\le p-1<p$, Lucas'
theorem~\cite{L} gives $\binom{p^2-1}{n}\equiv(-1)^{n}$ and
$(-1)^{p^2-1-n}\binom{p^2-1}{n}\equiv1$ for $0\le n\le p^2-1$.
 Hence \eqref{eq:newton} gives
\[
   \Delta^{p^2-1}(X_1^{2p-1})=\sum_{n=0}^{p^2-1}\sigma^{n}(X_1^{2p-1})
   =\sum_{n=0}^{p^2-1}(X_1+P_n)^{2p-1}
   =\sum_{a=0}^{2p-1}\binom{2p-1}{a}X_1^{a}\,S_{2p-1-a},
\]
where $S_k:=\sum_{n=0}^{p^2-1}P_n^{k}$. Writing $n=n_1p+n_0$, one has
$P_n=P_{n_0}-n_1$, since $\sum_{j\in\F_p}(X_2+j)^{p-1}=-1$ gives $P_{n+p}=P_n-1$.
With $\sum_{g\in\F_p}g^{i}=-1$ iff $(p-1)\mid i$, $i\ge1$ (and $0$ otherwise),
summation over each complete block yields
\[
   S_k=\sum_{n_0=0}^{p-1}\ \sum_{g\in\F_p}(P_{n_0}-g)^{k}
   =-\sum_{m\ge1}\binom{k}{(p-1)m}\,T_{k-(p-1)m},
   \qquad T_r:=\sum_{n_0=0}^{p-1}P_{n_0}^{r},
\]
where only the indices $m$ with $(p-1)m\le k$ occur.

We claim $S_k=0$ for $0\le k\le 2p-2$. In that range only $m=1$ and $m=2$ can
occur. For $m=1$, if $k\le p-2$ then $\binom{k}{p-1}=0$, and if $k=p-1$ the term is
$T_0=p\equiv0$. If $p\le k\le 2p-2$, then $k=k_0+1\cdot p$ with $0\le k_0\le p-2$, so
Lucas gives $\binom{k}{p-1}\equiv\binom{1}{0}\binom{k_0}{p-1}=0$.
 For $m=2$ the constraint
$2p-2\le k\le 2p-2$ forces $k=2p-2$, and the term is
$\binom{2p-2}{2p-2}T_0=T_0\equiv0$. This proves the claim.

Consequently every term with $a\ge1$ in the expansion above has index
$2p-1-a\le 2p-2$ and therefore vanishes, and the sum reduces to its $a=0$ term
$S_{2p-1}$. In $S_{2p-1}$ the surviving indices are $m=1$ and $m=2$, and from
$2p-1=(p-1)+1\cdot p$, $p-1=(p-1)+0\cdot p$ and $2p-2=(p-2)+1\cdot p$, Lucas gives
$\binom{2p-1}{p-1}\equiv\binom{1}{0}\binom{p-1}{p-1}=1$ and
$\binom{2p-1}{2p-2}\equiv\binom{1}{1}\binom{p-1}{p-2}=p-1\equiv-1$, whence
$S_{2p-1}=-\bigl(T_p-T_1\bigr)=T_1-T_p$. Finally the power-sum
identity $\sum_{j\in\F_p}(1+j)(Y+j)^{p-1}=Y-1$, together with
$(X_2+j)^{p(p-1)}=(X_2^{p}+j)^{p-1}$ (Frobenius), gives
\[
   T_1=\sum_{n_0}P_{n_0}=1-X_2,\qquad T_p=\sum_{n_0}P_{n_0}^{p}=1-X_2^{p},
\]
so that
\[
   \Delta^{p^2-1}(X_1^{2p-1})=S_{2p-1}=T_1-T_p=X_2^{p}-X_2\ \neq\ 0 .
\]
Therefore $\tau_K^{M_p}(X_1^{2p-1})=p^2$, whence $\tau_K(G_p)=p^2$, and
$|G_p|=|M_p|=p^2$.
\end{proof}

\begin{remark}\label{rem:as-invariant}
For any $F$ and any $g$ with $\tau_K^F(g)=m<\infty$, the last nonzero iterate
$\Delta_F^{m-1}(g)$ lies in $\ker\Delta_F$, the subring of $K[X]$ of polynomials
fixed by $\sigma_F$. In the present case $\Delta_{M_p}$ restricts on $K[X_2]$ to
$q(X_2)\mapsto q(X_2+1)-q(X_2)$, so $\ker\Delta_{M_p}\cap K[X_2]=K[X_2^{p}-X_2]$,
generated by the Artin--Schreier polynomial $\wp(X_2)=X_2^{p}-X_2$. The computed
value $\Delta^{p^2-1}(X_1^{2p-1})=\wp(X_2)$ is that generator.
\end{remark}

\section{Summary}\label{sec:disc}

Theorems~\ref{thm:tower} and~\ref{thm:Gp} describe the order and depth of Pascal
finite plane automorphisms. The order is a function of the depth
(Lemma~\ref{lem:order-depth}), and Dolgachev's theorem restricts it to the three
values $\{1,p,p^2\}$, so that order $p^2$ confines the depth to
$\{p+1,\dots,p^2\}$. Both nontrivial values are attained, the larger one sharply
in order (the Witt vectors, Proposition~\ref{prop:witt}) and in depth (the map
$G_p$, Theorem~\ref{thm:Gp}).

The depth is the finer invariant. It distinguishes maps of equal order, and it is
not a conjugacy invariant. By Lemma~\ref{lem:conj} it is computed from the
coordinates of the conjugator, and $\tau_K(M_p)=p+1$ while $\tau_K(G_p)=p^2$
(Theorem~\ref{thm:Gp}). At the bottom of the range, the quasi-translations are
exactly the maps of depth at most $2$, and so have order dividing $p$. The
converse fails, an order-$p$ map being subject only to $\tau_K\le p$.

The plane bound $\tau_K\le p^2$ has, by Proposition~\ref{prop:chains}, no
analogue in higher dimension. It rests on two independent two-dimensional
structures, each absent for $n\ge3$: the birational geometry of minimal rational
surfaces behind Dolgachev's theorem, and the Jung--van der Kulk amalgam behind
Proposition~\ref{prop:amalgam}. The latter has no counterpart in dimension
three, where the Nagata map~\cite{N} is wild.

\section*{Acknowledgements}

This research was supported by the AGH University of Krakow within subsidy of Polish Ministry of Science and Higher Education (grant no. 16.16.420.054).

\end{document}